\documentclass[11pt,reqno]{amsart} 
\usepackage[height=190mm,width=130mm]{geometry}
\usepackage{amsmath}
\usepackage{amssymb,latexsym}
\usepackage{amsrefs}
\usepackage[draft=true]{hyperref}
\usepackage{cite}
\usepackage{a4wide}
\usepackage{amscd}
\usepackage{graphics}
\usepackage{color}
\usepackage[all]{xy}

\theoremstyle{plain}
\newtheorem{theorem}{Theorem}
\newtheorem{lemma}{Lemma}
\newtheorem{corollary}{Corollary}
\newtheorem{proposition}{Proposition}

\theoremstyle{definition}
\newtheorem{definition}{Definition}

\theoremstyle{remark}
\newtheorem{remark}{Remark}
\newtheorem{example}{Example}
\newcommand{\Z}{\ensuremath{\mathbb{Z}}}   

\newcommand{\Q}{\ensuremath{\mathbb{Q}}}

\newcommand{\Hom}{\operatorname{Hom}}

\newcommand{\Ext}{\operatorname{Ext}} 

\newcommand{\Image}{\operatorname{Im}} 
\newcommand{\Ker}{\operatorname{Ker}}
\newcommand{\Rad}{\operatorname{Rad}}
\newcommand{\Soc}{\operatorname{Soc}}

\numberwithin{equation}{section} 

\begin{document}
\title[Max-Projective Modules]{MAX-PROJECTIVE MODULES}

\author{Yusuf Alag\"oz}

\address{Izmir Institute of Technology \\ Department of Mathematics\\ 35430 \\ Urla, \.{I}zmir\\ Turkey}

\curraddr{Siirt University\\ Department of Mathematics\\ Siirt \\ Turkey}

\email{yusufalagoz@iyte.edu.tr}

\author{ENG\.{I}N B\"uy\"uka\c{s}{\i}k}

\address{Izmir Institute of Technology \\ Department of Mathematics\\ 35430 \\ Urla, \.{I}zmir\\ Turkey}

\email{enginbuyukasik@iyte.edu.tr}

\begin{abstract}
A right $R$-module $M$ is called  \emph{max-projective}  provided that each homomorphism $f:M \to R/I$ where $I$ is any maximal right ideal, factors through the canonical projection $\pi : R \to R/I$. We call a ring $R$ right almost-$QF$ (resp. right max-$QF$) if every injective right $R$-module is $R$-projective (resp. max-projective). This paper attempts to understand the class of right almost-$QF$ (resp. right max-$QF$) rings. Among other results, we prove that a right Hereditary right Noetherian ring $R$ is right almost-$QF$ if and only if $R$ is right max-$QF$ if and only if $R=S\times T$ , where $S$ is semisimple Artinian and $T$ is right small. A right Hereditary ring is max-$QF$ if and only if every injective simple right $R$-module is projective. Furthermore, a commutative Noetherian ring $R$ is almost-$QF$ if and only if $R$ is max-$QF$ if and only if $R=A \times B$, where $A$ is $QF$ and $B$ is a small ring.
\end{abstract}



\maketitle

\section{Introduction and Preliminaries}
Throughout, $R$ will denote an associative ring with identity, and modules will be unital right $R$-modules, unless otherwise stated.  Let $M$ and $N$ be $R$-modules. $M$ is called $N$-projective (projective relative to $N$) if every $R$-homomorphism from $M$ into an image of $N$ can be lifted to an $R$-homomorphism from $M$ into $N$. $M$ is called $R$-projective if it is projective relative to the right $R$-module $R_{R}$. The module $M$ is called projective if $M$ is $N$-projective, for every $R$-module $N$. A right $R$-module $M$ is called  \emph{max-projective}  provided that each homomorphism $f:M \to R/I$ where $I$ is any maximal right ideal, factors through the canonical projection $\pi : R \to R/I$. This notion properly generalizes the notions  $R$-projective and rad-projective modules studied in \cite{radprojective}.

Characterizing rings by projectivity of some classes of their modules is a classical problem in ring and module theory. A result of Bass \cite[Theorem 28.4]{Anderson-Fuller:RingsandCategoriesofModules} states that a  ring $R$ is right perfect if and only if each flat right $R$-module is projective. On the other hand, the ring $R$ is $QF$ if and only if each injective right $R$-module is projective, \cite{FaithQF}. Recently, the notion of $R$-projectivity and its generalizations are considered  in \cite{testingforprojectivity,tauprojective,radprojective,almostperfect,testing}. The rings whose flat right $R$-modules are $R$-projective and max-projective are characterized in \cite{almostperfect,almostperfectrings} and \cite{maxprojective}, respectively.

We call a ring $R$ right almost-$QF$ (resp. right max-$QF$) in case all injective right $R$-modules are $R$-projective (resp. max-projective). Right almost $QF$-rings are max-$QF$. The ring of integers is almost-$QF$, since $\Hom(E, \Z/n\Z)=0$ for each injective $\Z$-module $E$.

In this paper, we investigate some properties of max-projective $R$-modules, and give some characterizations of almost-$QF$ and max-$QF$ rings.

We organize the paper as follows. In Section 2, some properties of   max-projective $R$-modules are investigated. We obtain that $R$-projectivity and max-projectivity coincide over the ring of integers and over right perfect rings. Characterizations of semiperfect, perfect and $QF$ rings in terms of max-projectivity are given. As an application, we show that a ring $R$ is right (semi)perfect if and only if every (finitely generated) right R-module has a max-projective cover if and only if every (simple) semisimple right $R$-module has a max-projective cover. By \cite[Lemma 2.1]{testingforprojectivity} any finitely generated $R$-projective right $R$-module is projective. This result is not true when $R$-projectivity is replaced with max-projectivity. We prove that if $R$ is either a semiperfect or nonsingular self-injective ring, then finitely generated max-projective right $R$-modules are projective. We show that any max-projective right $R$-module of finite length is projective.

In Section 3, we give some characterizations of almost-$QF$ and max-$QF$ rings.
Every right small ring is right max-$QF$, while a right small ring is right almost-$QF$ provided $R$ is right Hereditary or right Noetherian. A right Hereditary right Noetherian ring $R$ is right almost-$QF$ if and only if $R$ is right max-$QF$ if and only if $R=S\times T$ , where $S$ is a semisimple Artinian and $T$ is a right small ring. A right Hereditary ring $R$ is right max-$QF$ if and only if every simple injective right $R$-module is projective. A commutative Noetherian ring $R$ is almost-$QF$ if and only if $R$ is max-$QF$ if and only if $R=A \times B$, where $A$ is $QF$ and $B$ is a small ring. A right Noetherian local ring is almost-$QF$ if and only if $R$ is $QF$ or right small.

As usual, we denote by Mod$-R$ the category of right $R$-modules. For a module $M$, $E(M)$, $Z(M)$, $\Rad(M)$ and $Soc(M)$ denote the the injective hull, singular submodule, Jacobson radical and socle of $M$, respectively. The notation $K \ll M$ means that $K$ is a superfluous submodule of $M$ in the sense that $K + L\neq M$ for any proper submodule $L$ of $M$.

\section{Max-projective modules}
\begin{definition}
A right $R$-module $M$ is said to be \emph{max-projective} if for every epimorphism $f:R\rightarrow R/I$ with $I$ is a maximal right ideal of $R$, and every homomorphism $g:M\rightarrow R/I$, there exists a homomorphism $h:M\rightarrow R$ such that $fh=g$.

\end{definition}
\begin{example} \label{example}
\item[(a)] Every projective $R$-module is max-projective.
\item[(b)] The $\Z$-module $\Q$ is max-projective, since $\Hom(\Q,\,\Z_p)=0$ for each simple $\Z$-module $\Z_p$.
\item[(c)] Every simple max-projective $R$-module is projective. For if $S$ is a simple right $R$-module and $1_S: S\to S$ is the identity map, then  by max-projectivity of $S$ there is a homomorphism $f: S \to R$ such that $\pi f=1_S$, where $\pi:R\rightarrow S$ is the natural epimorphism. Then $R\cong K \oplus S$, and so $S$ is projective.
\item[(d)] Any $R$-module $M$ with $\Rad(M)=M$ is max-projective, since $M$ has no simple factors.

\end{example}

Given modules $M$ and $N$, $M$ is said to be $N$-subprojective if for every homomorphism $f:M\rightarrow N$ and for every epimorphism $g:B\rightarrow N$, there exists a homomorphism $h:M\rightarrow B$ such that $gh = f$ (see \cite{subprojective}).

\begin{lemma} \label{Rprojmaxproj}
For an $R$-module $M$, the following are equivalent.
\begin{enumerate}
\item $M$ is max-projective.
\item $M$ is $S$-subprojective for each simple $R$-module $S$.
\item For every epimorphism $f:N\rightarrow S$ with $S$ simple, and homomorphism $g:M\rightarrow S$, there exists a homomorphism $h:M\rightarrow N$ such that $fh=g$.

\end{enumerate}

\end{lemma}
\begin{proof}
$(2) \Leftrightarrow (3)$ By definition.
$(3)\Rightarrow (1)$ is clear.\\
$(1)\Rightarrow (3)$ Let $f:N\rightarrow S$ be an epimorphism with $S$ is simple $R$-module and $g:M\rightarrow S$ a homomorphism. Since $S$ is simple, there exists an epimorphism $\pi:R\rightarrow S$. By the hypothesis there exists a homomorphism $h:M\rightarrow R$ such that $\pi h=g$. Since $R$ is projective, there exists a homomorphism $h':R\rightarrow N$ such that $fh'=\pi$. Then $f(h'h)=\pi h=g$, and so $M$ is max-projective.
\end{proof}

We need the following result in the sequel.

\begin{lemma} \label{directsumofmaxprojectives}
The following conditions are true.
\begin{enumerate}
\item A direct sum $\oplus_{i\in I} A_{i}$ of modules is max-projective (resp. $R$-projective) if and only if each $A_{i}$ is max-projective (resp. $R$-projective).
\item If $0\rightarrow A\rightarrow B\rightarrow C\rightarrow 0$ is an exact sequence and $M$ is $B$-projective, then $M$ is projective relative to both $A$ and $C$.
\end{enumerate}
\end{lemma}
\begin{proof}
$(1)$ Since it is similar to the one provided in \cite[Proposition 16.10]{Anderson-Fuller:RingsandCategoriesofModules} for $R$-projective modules, the proof is omitted for max-projective modules.

$(2)$ is clear by \cite[Proposition 16.12]{Anderson-Fuller:RingsandCategoriesofModules}.
\end{proof}

\begin{corollary}
For a ring $R$, the following are equivalent.
\begin{enumerate}
\item $R$ is semisimple.
\item Every right $R$-module is max-projective.
\item Every finitely generated right $R$-module is max-projective.
\item Every cyclic right $R$-module is max-projective.
\item Every simple right $R$-module is max-projective.
\end{enumerate}
\end{corollary}

\begin{proof}
$(1)\Rightarrow (2) \Rightarrow (3) \Rightarrow (4) \Rightarrow (5)$ are clear.

$(5) \Rightarrow (1)$ Example \ref{example}$(c)$ and the hypothesis implies that each simple right $R$-module is projective. Thus $R$ is semisimple.
\end{proof}

In \cite{radprojective}, the module $M$ is called \emph{rad-projective} if, for any epimorphism $\sigma:R\rightarrow K$ where $K$ is an image of $R/J(R)$ and any homomorphism $f:M\rightarrow K$, there exists a homomorphism $g:M\rightarrow R$ such that $f=\sigma g$. We have the following implications:\\
projective $\Rightarrow$ R-projective $\Rightarrow$ rad-projective $\Rightarrow$ max-projective

\begin{proposition} \label{maxprojectivesemilocal}
Let $R$ be a semilocal ring and $M$ an $R$-module. Then the following are equivalent.
\begin{enumerate}
\item $M$ is rad-projective.
\item $M$ is max-projective.
\item Every homomorphism $f:M\rightarrow R/J(R)$ can be lifted to a homomorphism $g:M\rightarrow R$.
\end{enumerate}
\end{proposition}
\begin{proof}
$(1)\Rightarrow (2)$ Clear. $(3)\Rightarrow (1)$ By \cite[Proposition 3.14]{tauprojective}.

$(2)\Rightarrow (3)$ Since $R/J(R)$ is semisimple, $R/J(R)=\oplus^{n}_{i=1}K_{i}$, with each $K_{i}$ simple as an $R$-module. Let $\pi_{i}:\oplus^{n}_{i=1}K_{i}\rightarrow K_{i}$, and $\pi:R\rightarrow \oplus^{n}_{i=1}K_{i}$. Set $h:=\pi_{i}\pi$. By the hypothesis, there exists a homomorphism $g:M\rightarrow R$ such that $hg=\pi_{i}f$. Since $R/J(R)$ is semisimple, $\pi_{i}$ splits and there exists a homomorphism $\varepsilon_{i}:K_{i}\rightarrow \oplus^{n}_{i=1}K_{i}$ such that $\varepsilon_{i}\pi_{i}=1_{R/J(R)}$. Then, $\pi g=\varepsilon_{i} hg=\varepsilon_{i}\pi_{i} f=f$.
\end{proof}

In the next Proposition we provide a sufficient condition for an $R$-module to be max-projective. We establish a converse in the case of self-injective rings.
\begin{proposition} \label{p-testing}
If $M$ is a right $R$-module such that $Ext^{1}_{R}(M,I)=0$ for every maximal right ideal $I$ of $R$, then $M$ is max-projective. The converse is true when $R$ is a right self-injective ring.
\end{proposition}
\begin{proof}
By applying $\Hom(M,-)$ to the short exact sequence $0\rightarrow I\rightarrow R\rightarrow R/I\rightarrow 0$, with $I$ being a maximal right ideal of $R$, we obtain the following exact sequence:\\
$0\rightarrow \Hom(M,I)\rightarrow \Hom(M,R)\rightarrow \Hom(M,R/I)\rightarrow Ext^{1}_{R}(M,I)\rightarrow Ext^{1}_{R}(M,R)\rightarrow ...$. If $Ext^{1}_{R}(M,I)=0$ for every maximal right ideal $I$ of $R$, it follows that $M$ is max-projective. Conversely, since $R$ is right self injective, $Ext^{1}_{R}(M,R)=0$. If $M$ is a max-projective right $R$-module, then the map $\Hom(M,R)\rightarrow \Hom(M,R/I)$ is onto, and so $Ext^{1}_{R}(M,I)=0$ for any maximal right ideal $I$ of $R$.
\end{proof}

\begin{proposition} \label{subprojective}
Let $0\rightarrow A\rightarrow B\rightarrow C\rightarrow 0$ be a short exact sequence. If $M$ is $A$-subprojective and $C$-subprojective, then $M$ is $B$-subprojective.
\end{proposition}
\begin{proof}
Let $\gamma:F\rightarrow B$ be an epimorphism with $F$ projective. Then using the pullback diagram of $\gamma:F\rightarrow B$ and $\beta:A\rightarrow B$, and applying $\Hom(M, -)$, we get a commutative diagram with exact rows and columns:

\[\xymatrix@1{
&   & 0 \ar [d] &0 \ar [d]&
\\
 0 \ar [r] & Hom(M,K) \ar [r]  \ar [d] & Hom(M,X) \ar [r]^{\theta} \ar [d] & Hom(M,A) \ar [r] \ar [d]^{\beta^{*}} & 0
\\
 0 \ar [r] & Hom(M,K) \ar [r]   & Hom(M,F) \ar [r]^{\gamma^{*}} \ar [d] & Hom(M,B) \ar [r] \ar [d] & 0
\\
 &   & Hom(M,C) \ar [r]^{\phi} \ar [d] & Hom(M,C) \ar [d] &
 \\
 &   & 0  &0 & }\]
 Since $M$ is $A$-subprojective and $C$-subprojective, ${\theta}$ and ${\phi}$ are epic. Hence, ${\gamma^{*}}$ is epic by \cite[Five Lemma 3.15]{Anderson-Fuller:RingsandCategoriesofModules}.

\end{proof}
\begin{proposition}\label{prop:finitelength}
Let $M$ be an $R$-module. $M$ is max-projective if and only if $M$ is $N$-subprojective for any $R$-module $N$ with composition length $cl(N)< \infty$.
\end{proposition}
\begin{proof}
Let $M$ be a max-projective $R$-module and $N$ be an $R$-module with $cl(N)=n$. Then there exists a composition series $0=S_{0}\subset S_{1} ... \subset S_{n}=N$ with each composition factor $S_{i+1}/S_{i}$ simple. Consider the short exact sequence $0\rightarrow S_{1}\rightarrow S_{2}\rightarrow S_{2}/S_{1}\rightarrow 0$. Since $M$ is max-projective, by Lemma \ref{Rprojmaxproj}, $M$ is $S_{1}$-subprojective and $S_{2}/S_{1}$-subprojective. So, by Proposition \ref{subprojective}, $M$ is $S_{2}$-subprojective. Continuing in this way, $M$ is $S_{i}$-subprojective for each $0\leq i\leq n$. Hence, $M$ is $N$-subprojective. Conversely, since each simple right $R$-module has finite length, $M$ is max-projective by Lemma \ref{Rprojmaxproj}.
\end{proof}

\begin{corollary}
A $\Z$-module $M$ is max-projective if and only if $M$ is $\Z$-projective.
\end{corollary}

\begin{proof} By the Fundamental Theorem of Abelian Groups, a cyclic  $\Z$-module $M$ is isomorphic either to $\Z$ or  to a finite direct sum of $\Z$-modules of finite length. Now the proof is clear by Proposition \ref{prop:finitelength}.
\end{proof}

\begin{corollary}\label{cor:finitelengthmaxproj.isproj.} Let $M$ be an $R$-module with finite composition length. If $M$ is max-projective, then it is projective.
\end{corollary}

\begin{proof} Let $f: R^n \to M$ be an epimorphism. The module $M$ is $M$-subprojective by Proposition \ref{prop:finitelength}. That is, there is a homomorphism $g: M \to R^n$ such that $1_M=fg$. Thus the map $f$ splits, and so $M$ is projective.
\end{proof}

Submodules of max-projective $R$-modules need not be max-projective. Consider the ring $R=\Z/ p^2\Z$, for some prime integer $p$. $R$ is max-projective, whereas the simple ideal $pR$ is not max-projective, since  the epimorphism $R \to pR \to 0$ does not split.

Recall that a ring $R$ is called right \emph{$V$-ring} (resp. \emph{right $GV$-ring}) if all simple (resp. all singular simple) right $R$-modules are injective.
\begin{proposition}
Consider the following conditions for a ring $R$:
\begin{enumerate} \label{gvring}
\item $R$ is a right $GV$-ring.
\item Submodules of  max-projective right $R$-modules are max-projective.
\item Submodules of  projective right $R$-modules are max-projective.
\item Every right ideal of $R$ is max-projective.
\end{enumerate}

Then, $(1)\Rightarrow (2)\Rightarrow (3)\Rightarrow (4)$. Also, if $R$ is a right self injective ring, then $(4)\Rightarrow (1)$.
\end{proposition}
\begin{proof}
$(1)\Rightarrow (2)$ Let $N$ be a submodule of a max-projective right $R$-module $M$. Consider the following diagram:
$$\xymatrix{0 \ar[r]  &N \ar[d]^{f} \ar[r]^{i}&M\\
R\ar[r]^{\pi} & S \ar[r] & 0
} $$ where $S$ is a simple right $R$-module, $i:N\rightarrow M$ is the inclusion map and $\pi:R\rightarrow S$ is the canonical quotient map. Since the simple module $S$ is either projective or singular, the former implies $\pi:R\rightarrow S$ splits and there exists a homomorphism $\varepsilon:S\rightarrow R$ such that $\varepsilon \pi=1_{R}$. In the latter one, $S$ is singular, so it is injective by the hypothesis. Thus, there is a homomorphism $g:M\rightarrow S$ such that $gi=f$. Since $M$ is max-projective, there is a homomorphism $h:M\rightarrow R$ such that $\pi h=g$. Hence, $\pi(hi)=gi=f$. In either case, there exists a homomorphism from $N$ to $R$ that makes the diagram commute. This implies that $N$ is max-projective.

$(2)\Rightarrow (3)\Rightarrow (4)$ Clear.
$(4)\Rightarrow (1)$ Let $I$ be a right ideal of $R$ and $J$ a maximal right ideal of $R$. Consider the following diagram:
 $$\xymatrix{0 \ar[r]  &I \ar[d]^{f} \ar[r]^{i}&R\\
R\ar[r]^{\pi} & R/J \ar[r] & 0
} $$ where $R/J$ is a simple right $R$-module, $i:I\rightarrow R$ is the inclusion map and $\pi:R\rightarrow R/J$ is the canonical quotient map. Since $I$ is max-projective, there is a homomorphism $h:I\rightarrow R$ such that $\pi h=f$. Since $R$ is injective, there exists a homomorphism $\lambda:R\rightarrow R$ such that $\lambda i=h$. Now the map $\beta=\pi\lambda:R\rightarrow R/J$ satisfies $\beta i=\pi\lambda i=\pi h=f$, as required.
\end{proof}

\begin{proposition}Let $R$ be a commutative or semilocal ring. Then pure submodules of max-projective $R$-modules are max-projective.
\end{proposition}

\begin{proof} Let $M$ be a max-projective (right) module and $N$ a pure submodule of $M$. Let $S$ be a simple (right) module and $f: N \to S$ be a homomorphism. Since  $S$ is pure-injective and $N$ is a pure submodule of $M$, there is $g:M\to S$ such that $gi=f$, where $i: N \to M$ is the inclusion map. By max-projectivity of $M$, there is a homomorphism $h:M \to R$ such that $g=\pi h$, where $\pi : R \to S$ is the natural epimorphism. Now we have $f=gi=\pi hi$, i.e. $hi: N\to R$ lifts $f$. This proves that $N$ is max-projective.
\end{proof}

\begin{lemma}\label{factormaxprojective}
Let $R$ be a ring and $\tau$ be a preradical with $\tau(R)=0$. If $M$ is a max-projective $R$-module, then $M/\tau(M)$ is max-projective.
\end{lemma}
\begin{proof}
Let $M$ be a max-projective $R$-module and $f:M/\tau(M)\rightarrow S$ a homomorphism with $S$ simple $R$-module. Consider the following diagram:
\[
\xymatrix{
 M\ar[r]^\pi & M/\tau(M)  \ar[d]^{f} \\
R \ar[r]^\eta & S \ar[r] & 0
}
\]
Since M is max-projective, there exists a homomorphism $g:M\rightarrow R$ such that $f\pi=\eta g$. Since $g(\tau(M))\subseteq \tau(R)=0$, $\tau(M))\subseteq \Ker(g)$, and so there exists a homomorphism $h:M/\tau(M)\rightarrow R$ such that $h\pi=g$. Now, since $\eta h\pi=\eta g=f\pi$ and $\pi$ is an epimorphism, $\eta h=f$, and so $M/\tau(M)$ is max-projective.
\end{proof}

\begin{remark}
Recall that any finitely generated $R$-projective module is projective, \cite[Lemma 2.1]{testingforprojectivity}. This is not true for max-projective modules in general. Let $R$ be a right $V$-ring which is not right semihereditary. Then $R$ has a finitely generated right ideal  which is not projective. By Proposition \ref{gvring}, each right ideal of $R$ is max-projective.
\end{remark}

\begin{proposition} \label{fgmax}
Let $R$ be a right nonsingular right self-injective ring. Every finitely generated max-projective right $R$-module is projective.
\end{proposition}
\begin{proof}
Let $M$ be a finitely generated max-projective right $R$-module. As $R$ is a right nonsingular ring, by Lemma \ref{factormaxprojective}, $M/Z(M)$ is max-projective. Since $M/Z(M)$ is finitely generated, there exists an epimorphism $f:F\rightarrow M/Z(M)$ such that $F$ is finitely generated free. This means $\Ker(f)$ is closed in $F$. By the injectivity of $F$, $\Ker(f)$ is a direct summand of $F$, and so $M/Z(M)$ is projective. Then, $M=Z(M)\oplus K$ for some projective submodule $K$ of $M$. We claim that $Z(M)=0$. Assume to the contrary that $Z(M)\neq 0$. Since, $Z(M)$ is a finitely generated submodule of $M$, there exists a nonzero epimorphism $g:Z(M)\rightarrow S$ for some simple right $R$-module $S$. Then, by Lemma \ref{directsumofmaxprojectives}, $Z(M)$ is max-projective, and so there exists a nonzero homomorphism $h:Z(M)\rightarrow R$ such that $\pi h=g$, where $\pi:R\rightarrow S$ is the natural epimorphism. But then $h(Z(M))\subseteq Z(R_{R})=0$, a contradiction. Thus we must have $Z(M)=0$, whence $M$ is projective.
\end{proof}

A ring $R$ is called \emph{right max-ring} if every nonzero right $R$-module $M$ has a maximal submodule i.e. $\Rad (M) \neq M$.
\begin{proposition}\label{maxprojectiveradsmall}
The following conditions are true.
\begin{enumerate}
\item Over a semiperfect ring $R$, every max-projective right $R$-module with small radical is projective.
\item A ring $R$ is right perfect if and only if $R$ is semilocal and every max-projective right $R$-module is projective.
\end{enumerate}
\end{proposition}

\begin{proof}
$(1)$ Let $M$ be a max-projective right $R$-module with $\Rad(M)\ll M$. Since $R$ is semilocal, $M$ is rad-projective by Proposition \ref{maxprojectivesemilocal}. Hence $M$ is projective by \cite[Theorem 4.7]{tauprojective}.

$(2)$ Since over a right perfect ring $R$ every right $R$-module has  small radical, it follows from $(1)$ that every max-projective right $R$-module is projective. Conversely, assume that $R$ is semilocal and every max-projective right $R$-module is projective. Let $M$ be a nonzero right $R$-module. We claim that $\Rad(M)\neq M$. Assume to the contrary that $M$ has no maximal submodule, i.e. $\Rad(M)= M$. Since $\Hom(M,S)=0$ for any simple right $R$-module, $M$ is max-projective. Thus $M$ is projective, by the hypothesis. Since projective modules have maximal submodules, this is a contradiction. Hence, every right $R$-module has a maximal submodule. Since $R$ is semilocal, $R$ is right perfect by \cite[Theorem 28.4]{Anderson-Fuller:RingsandCategoriesofModules}.
\end{proof}
Recall that if $R$ is a right perfect ring, every $R$-projective right $R$-module is projective, \cite{testing}. Thus the following result follows from Proposition \ref{maxprojectiveradsmall}(2).

\begin{corollary} Let $R$ be a right perfect ring and $M$ be a right $R$-module. Then the following are equivalent.
 \begin{enumerate}
 \item $M$ is projective.
\item $M$ is $R$-projective.
\item $M$ is max-projective.
\end{enumerate}
\end{corollary}

The following Remark is an example of a right nonperfect ring $R$ such that every max-projective module is $R$-projective.
\begin{remark} \label{max-projectivesR-projective}
Let $K$ be a field, and $R$ the subalgebra of $K^{\omega}$ consisting of all eventually constant sequences in $K^{\omega}$. For each $i<\omega$, we let $e_{i}$ be the idempotent in $K^{\omega}$ whose ith component is $1$ and all the other components are $0$. Notice that $\{e_{i} : i<\omega \}$ a set of pairwise orthogonal idempotents in $R$, so $R$ is not perfect, \cite[Lemma 2.3]{trlifaj}. By \cite[Lemma 2.3 and Lemma 2.4]{trlifaj}, $R/Soc(R)$ is simple $R$-module and a module $M$ is $R$-projective if and only if it is projective with respect to the projection $\pi:R\rightarrow R/Soc(R)$. Thus, an $R$-module $M$ is max-projective if and only if $M$ is $R$-projective.
\end{remark}
The following Corollary follows from \cite[Theorem 3.3]{trlifaj} and Remark \ref{max-projectivesR-projective}.

\begin{corollary}Let $K$ be a field of cardinality $\leq 2^{\omega}$ and $R$ the subalgebra of $K^{\omega}$ consisting of all eventually constant sequences in $K^{\omega}$. Assume G\"odel's Axiom of Constructibility $(V = L)$. Then all max-projective $R$-modules are projective.
\end{corollary}

\begin{lemma}\label{maxprojectivefactor}
If $M_{R}$ is max-projective and $\bar{R}=R/J(R)$, then $(M/\Rad(M))_{\bar{R}}$ is max-projective.
\end{lemma}
\begin{proof}
Let $\pi:\bar{R}_{\bar{R}}\rightarrow K_{\bar{R}}$ be an $\bar{R}$-epimorphism with $K_{\bar{R}}$ simple $\bar{R}$-module. Consider the following diagram:
\[
\xymatrix{
 M\ar[r]^{\eta} & M/Rad(M)  \ar[d]^{f} \\
\bar{R}_{\bar{R}} \ar[r]^\pi & K_{\bar{R}} \ar[r] & 0
}
\]
Since $M$ is max-projective, there exists a homomorphism $\lambda:M_{R}\rightarrow \bar{R}_{R}$ such that $\pi\lambda=f\eta$. Since $\lambda(\Rad(M))\subseteq \Rad(R/J(R))=0$, $\Rad(M)\subseteq \Ker(\lambda)$, and so there exists a homomorphism $\delta:(M/\Rad(M))_{R}\rightarrow \bar{R}_{R}$ such that $\delta\eta=\lambda$. Now, since $\pi\delta\eta=\pi\lambda=f\eta$ and $\eta$ is an epimorphism, $\pi\delta=f$, and so $(M/\Rad(M))_{\bar{R}}$ is a max-projective $\bar{R}$-module.
\end{proof}

It is well-known that a ring $R$ is \emph{semiperfect} if and only if every simple $R$-module has a projective cover. In the next Proposition, we extend this result by replacing projective covers with max-projective covers. Let $R$ be a ring and $\Omega$ be a class of right $R$-modules which is closed under isomorphisms. A homomorphism $f:P \rightarrow M $ is called an $\Omega$-cover of the right $R$-module $M$, if $P \in \Omega $ and $f$ is an epimorphism with small kernel. That is to say, if $\Omega$ is the class of max-projective right $R$-modules, the homomorphism $f:P \rightarrow M $ is called max-projective cover of $M$. With the help of an argument similar to the one provided in \cite[Theorem 18]{radprojective}, we can establish the next Proposition.
\begin{proposition}\label{maxprojectivecover}
For a ring $R$, the following are equivalent.
\begin{enumerate}
\item $R$ is semiperfect.
\item Every finitely generated right $R$-module has a max-projective cover.
\item Every cyclic right $R$-module has a max-projective cover.
\item Every simple right $R$-module has a max-projective cover.
\end{enumerate}
\end{proposition}
\begin{proof}
$(1)\Rightarrow (2)\Rightarrow (3)\Rightarrow (4)$ are clear.\\
$(4)\Rightarrow (1)$ We first show that $\bar{R} = R/J(R)$ is a semisimple ring. Let $S$ be a simple right $\bar{R}$-module. By the hypothesis $S_{R}$ has a max-projective cover $P_{R}$, say $f:P\rightarrow S$ with $Rad(P)=\Ker(f) \ll P$. Since $S$ is simple and $P/\Rad(P) \cong S$, $P/\Rad(P)$ is a simple $\bar{R}$-module. So, $(P/\Rad(P))_{\bar{R}} $ is max-projective by Lemma \ref{maxprojectivefactor}, whence $(P/\Rad(P))_{\bar{R}}$ is projective. Consider the map $\tilde{f}:P/\Rad(P)\rightarrow S$. This map induces an isomorphism. Since $P/\Rad(P)$ is projective $\bar{R}$-module, $P/\Rad(P)$ is the projective cover of $S_{\bar{R}}$. Hence, $\bar{R}$ is a semiperfect ring. Therefore, $\bar{R}$ is semisimple as an $\bar{R}$-module, and hence semisimple as an $R$-module. Write $\bar{R}= R/J(R)=\oplus^{n}_{i=1}K_{i}$, with each $K_{i}$ simple as a right $R$-module, and let $L_{i}$ be a max-projective cover of $K_{i}$,  $1 \leq i \leq n$, as right $R$-modules. Now, in order to prove that R is a semiperfect ring, it is enough to show that each $L_{i}$, $1 \leq i \leq n$, is projective as a right $R$-module. Clearly, $L=\oplus^{n}_{i=1}L_{i}$, as a right R-module, is a max-projective cover of $\bar{R}_{R}$. Consider the diagram
\[
\xymatrix{
  & L_{R} \ar@{.>}[dl]_{\exists g} \ar[d]^{f} \\
R_{R} \ar[r]^\pi & \bar{R}_{R} \ar[r] & 0
}
\]
with $f$ being the max-projective cover of $\bar{R}_{R}$, and $\pi$ the canonical $R$-epimorphism. By the max-projectivity of $L_{R}$, $f$ can be lifted to a map $g:L_{R}\rightarrow R_{R}$ such that $\pi g=f$. Since $R=\Image(g)+ J(R)$ and $J(R)\ll R$ we infer that $R=\Image(g)$ and $g$ is onto. By the projectivity of $R$, the map $g$ splits and $L_{R}=\Ker(g)\oplus A$ for a submodule $A$ of $L_{R}$. Since $\Ker(g)\subseteq \Ker(f)\ll L_{R}$, $\Ker(g)=0$ and $L_{R}\cong R_{R}$ is projective. Therefore, each
$L_{i}$, $1 \leq i \leq n$, is projective as a right $R$-module, and $R$ is semiperfect.
\end{proof}

\begin{proposition}\label{maxprojectiveperfect}
For a ring $R$, the following conditions are equivalent.
\begin{enumerate}
\item $R$ is right perfect.
\item Every right $R$-module has a max-projective cover.
\item Every semisimple right $R$-module has a max-projective cover.
\end{enumerate}
\end{proposition}
\begin{proof}
$(1)\Rightarrow (2)\Rightarrow (3)$ are clear.

$(3)\Rightarrow (1)$ By Proposition \ref{maxprojectivecover}, $R$ is a semiperfect ring. Let $M$ be a semisimple right $R$-module and $f:P\rightarrow M$ be a max-projective cover of $M$. Since $Rad(P)=\Ker(f) \ll P$, $P$ is projective by Proposition \ref{maxprojectiveradsmall}(1). Thus every semisimple right $R$-module has a projective cover, and so $R$ is right perfect.
\end{proof}

Let $R$ be any ring and $M$ be an $R$-module. A submodule $N$ of $M$ is called radical submodule  if $N$ has no maximal submodules, i.e. $N = \Rad(N)$. By $P(M)$ we denote the sum of all radical submodules of a module $M$. For any module $M$, $P(M)$ is the largest radical submodule of $M$, and so $\Rad(P(M))=P(M)$. Moreover, $P$ is an idempotent radical with $P(M)\subseteq \Rad(M)$ and $P(M/P(M))=0$, (see \cite{radsupp}).

In \cite[Lemma 1]{primehereditarynoether}, the authors prove that over a right nonsingular right $V$-ring, max-projective right $R$-modules are nonsingular.
Regarding the converse of this fact, we have the following.

\begin{proposition}
If every max-projective right $R$-module is nonsingular, then $R$ is right nonsingular and right max-ring.
\end{proposition}
\begin{proof}
Clearly the ring $R$ is right nonsingular. If $R$ is right $V$-ring, then $\Rad (M)=0$ for any right $R$-module $M$. Thus $R$ is a max-ring. Suppose $R$ is not right $V$-ring and let $S$ be a noninjective simple right $R$-module. We shall first see that $P(E(S))=0$. Suppose $\Rad(P(E(S)))=P(E(S))\neq 0$. Then $P(E(S))/S$  is singular. Furthermore, since $\Rad(P(E(S))/S)=P(E(S))/S$, $P(E(S))/S$ is max-projective. This contradicts with the hypothesis. Therefore, for every simple right $R$-module $S$, $P(E(S))=0$. Let $M$ be a nonzero right $R$-module. We claim that $\Rad(M)\neq M$. Assume to the contrary that $\Rad(M)=M$. Let $0\neq x \in M$ and $K$ be a maximal submodule of $xR$. Then the simple right $R$-module $S=xR/K$ is noninjective, because $S$ small. Now, the obvious map $xR\to E(S)$ extends to a nonzero map $f:M\rightarrow E(S)$. Since $P(\Image (f))\subseteq P(E(S))=0$, $P(M/\Ker (f))=0$. This contradicts with $P(M)=M$. Hence $\Rad (M)\neq M$ for every right $R$-module $M$, and so $R$ is a right max-ring.
 \end{proof}

\begin{corollary}
For a ring $R$, the following are equivalent.
\begin{enumerate}
\item $R$ is semilocal and every max-projective right $R$-module is nonsingular.
\item $R$ is  right perfect and right nonsingular.
\end{enumerate}
\end{corollary}

\section{Almost-$QF$ and max-$QF$ rings}

Recall that a ring $R$ is \emph{$QF$} if and only if every injective (right) $R$-module is projective (see, \cite{FaithQF}). We slightly weaken this condition and obtain the following definition.
\begin{definition}
A ring $R$ is called right \emph{almost-$QF$} if every injective right $R$-module is $R$-projective. We call $R$ \emph{right max-$QF$}, if every injective right $R$-module is max-projective. Left almost-$QF$ and left max-$QF$ rings are defined similarly.
\end{definition}
Clearly, we have the following inclusion relationship:\\

\{$QF$ rings\} $\subseteq$ \{right almost-$QF$ rings\} $\subseteq$ \{ right max-$QF$ rings\}.\\

\begin{example} The ring of integers $\Z$, is a right almost-$QF$ but not $QF$: For every injective $\Z$-module $E$, we have $\Rad (E)=E$. Thus $\Hom (E, \Z /n\Z)=0$, for each cyclic $\Z$-module $\Z /n\Z$. This means that each injective $\Z$-module is $\Z$-projective,and so $\Z$ is almost-$QF$.
\end{example}

\begin{remark}
Sandomierski \cite{testing} proved that if $R$ is a right perfect ring, then every $R$-projective right module is projective. Thus a ring $R$ is right perfect and right almost-$QF$ if and only if $R$ is $QF$.
\end{remark}

\begin{proposition}
Let $R$ and $S$ be Morita equivalent rings. Then, $R$ is right almost-$QF$ if and only if $S$ is  right almost-$QF$.
\end{proposition}
\begin{proof}
An R-module $M$ is $R$-projective if and only if $M$ is $N$-projective for any finitely generated projective $R$-module $N$. Now, by \cite[propositions 21.6 and 21.8 ]{Anderson-Fuller:RingsandCategoriesofModules}, since injectivity, relative projectivity and being finitely generated are preserved by Morita equivalence, the proof is clear.
\end{proof}

\begin{lemma} \label{product}
Let $R_{1}$ and $R_{2}$ be rings. Then $R=R_{1}\times R_{2}$ is right almost-$QF$ (resp. right max-$QF$) if and only if $R_{1}$ and $R_{2}$ are both right almost-$QF$ (resp. right max-$QF$).
\end{lemma}
\begin{proof}
Let $M$ be an injective right $R_{1}$-module. Then $M$ is an injective right $R$-module, as well as an $R$-projective module by the hypothesis. Hence, by Lemma \ref{directsumofmaxprojectives}, $M$ is $R_{1}$-projective, and so $R_{1}$ is right almost-$QF$. Similarly, $R_{2}$ is right almost-$QF$. Conversely, let $M$ be an injective right $R$-module. Since we have the decomposition $M=MR_{1}\oplus MR_{2}$, $MR_{1}$ is an injective right $R$-module, whence it is an injective right $R_{1}$-module. On the other hand, since $(MR_{2})R_{1}=0$, $MR_{2}$ is an $R_{1}$-module, and so it is an injective $R_{1}$-module. This means that $MR_{1}$ and $MR_{2}$ are $R_{1}$-projective by the hypothesis. Then, by Lemma \ref{directsumofmaxprojectives}, $M=MR_{1}\oplus MR_{2}$ is $R_{1}$-projective. Similarly, $M$ is $R_{2}$-projective. Therefore, $M$ is $R$-projective. Since it is similar to the one provided for almost-$QF$ rings, the proof is omitted for max-$QF$ rings.
\end{proof}
\begin{proposition}  \label{injind} Let $R$ be a right Hereditary ring and $E$ be an indecomposable injective right $R$-module. Then the following are equivalent.
 \begin{enumerate}
\item $E$ is $R$-projective.
\item $E$ is max-projective.
\item Either $\Rad(E)=E$ or $E$ is projective.
\end{enumerate}
\end{proposition}
\begin{proof}
 $(1)\Rightarrow (2)$ Clear.

 $(2)\Rightarrow (3)$ Assume that $\Rad(E)\neq E$. Then $E$ has a simple factor module isomorphic to $R/I$. Let $ f:E\rightarrow R/I$ be a nonzero homomorphism. Since $E$ is max-projective, there exists a homomorphism $g:E\rightarrow R$ such that $\Image(g) \neq 0$. By the fact that $R$ is right Hereditary, $\Image(g)$ is projective, whence $E\cong \Image(g)\oplus K$ for some right $R$-module $K$. Since $E$ is indecomposable, either $K = 0$ or $\Image(g) = 0$, where the latter case implies that $g=0$ which is a contradiction. In the former case $K = 0$, implying that $E$ is projective.\\
$(3)\Rightarrow (1)$ Conversely, if $E$ is projective then $E$ is clearly $R$-projective. Now suppose $\Rad(E)=E$ and let $f:E\rightarrow R/I$ be a homomorphism. Then $f(E)=f(\Rad(E)) \subseteq \Rad(R/I) \ll R/I$. Moreover $f(E)$ is a direct summand of $R/I$ since $R$ is right Hereditary. Therefore $f(E)=0$, and so $f$ can be lifted to $R$.
\end{proof}
\begin{lemma}\label{small} (See \cite[3.3]{smallring}) For a ring $R$ the following are equivalent.
\begin{enumerate}
\item $R$ is a right small ring.
\item $\Rad(E)= E$ for every injective right $R$-module $E$.
\item $\Rad(E(R))= E(R)$.
\end{enumerate}
\end{lemma}
\begin{corollary}
If $R$ is a right Hereditary right small ring, then $R$ is right almost-$QF$.
\end{corollary}

\begin{proposition}
If $R$ is a right semihereditary right small ring, then $\Hom(E,R)=0$, for any injective right $R$-module $E$. In particular, $R$ is right almost-$QF$ if and only if $\Hom(E,R/I)=0$ for any right ideal $I$ of $R$.
\end{proposition}
\begin{proof}
Let $E$ be an injective right $R$-module and $f\in \Hom(E,R)$. Then $f(E)=f(\Rad(E))\subseteq J(R)$. Since $R$ is right semihereditary, $f(E)$ is absolutely pure. This means that $R/f(E)$ is flat by \cite[Corollary 4.86]{Lam:LecturesOnModulesAndRings}. Then, by \cite[\S4 Exercise 20]{Lam:LecturesOnModulesAndRings}, $f(E)=0$, i.e. $\Hom(E,R)=0$. Hence, the rest is clear.
\end{proof}
Recall that by Example \ref{example}$(d)$, any  right small ring  $R$  is right max-$QF$. Moreover, if $R$ is right Noetherian, we have the following.
\begin{proposition} \label{smallalmostQF}
If $R$ is a right Noetherian and right small ring, then $R$ is right almost-$QF$.
\end{proposition}
\begin{proof}
Let $E$ be an injective right $R$-module. Then, by Lemma \ref{small}, $\Rad(E)=E$. Now let $f:E\rightarrow R/I$ be a homomorphism for any right ideal $I$ of $R$. This implies that $f(E)\subseteq R/I$ and since $\Rad(E)=E$, we have $\Rad(f(E))=f(E)$. By the right Noetherian assumption, $R/I$ is a Noetherian right $R$-module and its submodule $f(E)$ is finitely generated, i.e. $\Rad(f(E))\neq f(E)$. Also since $\Rad(f(E))=f(E)$, this means that $f(E)=0$, whence $f:E\rightarrow R/I$ can be lifted to $R$. Consequently, $E$ is $R$-projective.
\end{proof}

\begin{theorem} Let $R$ be a right Hereditary and right Noetherian ring. Then the following are equivalent.
\begin{enumerate}
\item $R$ is right almost-$QF$.
\item $R$ is right  max-$QF$.
\item Every injective right $R$-module $E$ has a decomposition $E=A\oplus B$ where $\Rad(A)=A$ and $B$ is projective and semisimple.
\item $R=S\times T$, where $S$ is a semisimple Artinian ring and $T$ is a right small ring.
\end{enumerate}
\end{theorem}

\begin{proof}
$(1)\Rightarrow (2)$ Clear.\\
$(2)\Rightarrow (3)$ Let $E$ be an injective right $R$-module. Then $E$ has an indecomposable decomposition $E=\oplus_{i \in \Gamma} A_{i}$ where $A_{i}$'s are either projective or $\Rad(A_{i})=A_{i}$ by proposition \ref{injind}. Let $\Lambda=\{j\in \Gamma : A_{j} $\ is projective\}. So the decomposition of $E$ can be written as $E=(\oplus_{j\in \Lambda} A_{j}) \oplus (\oplus_{i\in \Gamma-\Lambda} A_{i})$. We claim that each $A_{j}$ is simple for $j\in \Lambda$. Since $A_{j}$ is projective for $j\in \Lambda$, $\Rad(A_{j})\neq A_{j}$. So there exists a simple factor $B_{j}$ of $A_{j}$ i.e. $B_{j}\cong A_{j}/N\cong R/I$ for some maximal submodule $N$ of $A_{j}$ and for some maximal right ideal $I$ of $R$. Since $B_{j}$ is injective, by $(2)$, the following diagram commutes.
\[
\xymatrix{
  & B_{j} \ar@{.>}[dl]_{g} \ar[d]^{f} \\
R \ar@{->>}[r]^h & R/I \ar[r] & 0
}
\]
With the Hereditary assumption on $R$, $\Image(g)\cong B_{j}$ is projective and so $A_{j}\cong N\oplus B_{j}$. However $A_{j}$ is indecomposable, whence $N=0$. Consequently, each $A_{j}$ is simple for $j\in \Lambda$.

$(3)\Rightarrow (1)$ Let $E$ be an injective right $R$-module. By the assumption, $E=A\oplus B$ where $\Rad(A)=A$ and $B$ is semisimple and projective. Since $B$ is $R$-projective, we only need to show that $A$ is $R$-projective. By the Noetherian assumption, the injective $R$-module $A$ has a decomposition  $A=\oplus _{i \in \Gamma} A_{i}$ where each $A_{i}$ is indecomposable injective with $\Rad(A_{i})=A_{i}$. Proposition \ref{injind} implies that each $A_{i}$ is $R$-projective, whence $A$ is $R$-projective by Lemma \ref{directsumofmaxprojectives}. Therefore, $M=A\oplus B$ is $R$-projective by Lemma \ref{directsumofmaxprojectives}.

$(2)\Rightarrow (4)$ Let $S$ be the sum of minimal injective right ideals of $R$. Then $S$ is injective since $R$ is right Noetherian. Thus we have the decomposition $R=S\oplus T$ for some right ideal $T$ of $R$ such that $\Soc(S)=S$ and $T$ has no simple injective submodule.
If $f:S\rightarrow T$ is a nonzero homomorphism, then $f(\Soc(S))= f(S) \subseteq \Soc(T)$, where $f(S)$ is injective by the Hereditary assumption, and so $\Soc(T)$ contains a semisimple injective direct summand $f(S)$. This means that $f(S)=0$, a contradiction. Thus, we have $\Hom(S,T)=0$ , and so $S$ is a two sided ideal. On the other hand, if $g:T\rightarrow S$ is a nonzero homomorphism, then $T/\Ker(g)\cong \Image(g) \subseteq S$, and so $ T/\Ker(g)$ is projective by Hereditary assumption. Also since $S$ is a semisimple injective $R$-module, $ T/\Ker(g)$ is semisimple injective, whence $ K/\Ker(g)$ is semisimple injective for any maximal submodule $ K/\Ker(g)$ of $ T/\Ker(g)$. This implies that $T/\Ker(g)\cong K/\Ker(g) \oplus T/K$. Then the simple $R$-module $T/K$ is injective and projective, and so $T$ contains an isomorphic copy of a simple injective $R$-module $T/K$, yielding a contradiction. Therefore, $\Hom(T,S)=0$, and so $T$ is a two sided ideal. Consequently, $R=S\oplus T$ is a ring decomposition. Now let $E(T)$ be the injective hull of $T$ as an $R$-module. The injective hull $E(T)$ is also a $T$-module by the fact that $E(T)S=0$. We claim that $\Rad(E(T))=E(T)$. Suppose the contrary and let $K$ be a maximal submodule of $E(T)$. Then $E(T)/K$ is injective by the Hereditary assumption and it is max-projective by (2). Since $E(T)/K$ is a simple right $R$-module, it is isomorphic to $R/I$ for some maximal right ideal $I$ of $R$, and so $R/I$ is injective. Then, the isomorphism $\alpha: E(T)/K \rightarrow R/I$ lifts to $\beta:E(T)/K \rightarrow R$ i.e. the following diagram commutes.
\[
\xymatrix{
  & E(T)/K \ar@{.>}[dl]_{\beta} \ar[d]^{\alpha} \\
R \ar[r]^h & R/I \ar[r] & 0
}
\] Since $\beta$ is monic and $E(T)/K$ injective, $U=\beta(E(T)/K)$ is a direct summand of $R$. It is easy to see that $U$ is also a right $T$-module and so $U \subseteq  T$. On the other hand, since $U$ is minimal and injective, $U$ is also contained in $S$, a contradiction. So we must have $\Rad(E(T))=E(T)$, whence $T \ll E(T)$ by Lemma \ref{small}. This proves (4).

$(4)\Rightarrow (1)$ Clear, by Lemma \ref{product} and Proposition \ref{smallalmostQF}.

\end{proof}

\begin{theorem}
Let $R$ be a right Hereditary ring. Then the following are equivalent.
\begin{enumerate}
\item $R$ is right max-$QF$.
\item Every simple injective right $R$-module is projective.
\item Every singular injective right $R$-modules is $R$-projective.
\item Every singular injective right $R$-modules is max-projective.
\item $\Rad(E)=E$ for every singular injective right $R$-module $E$.
\item Every injective right $R$-module $E$ can be decomposed as $E=Z(E)\oplus F$ with $\Rad(Z(E))=Z(E)$.
\end{enumerate}
\end{theorem}

\begin{proof}
$(1)\Rightarrow (4)$, $(3)\Rightarrow (4)$ and $(6)\Rightarrow (5)$ are clear.

$(4)\Rightarrow (2)$ Let $S$ be a simple injective right $R$-module. We claim that $S$ is projective. Assume that $S$ is not projective. Then it is singular and injective. This implies, by our hypothesis that $S$ is max-projective, hence $S$ is projective, this is a contradiction. The conclusion now follows.

$(2)\Rightarrow (1)$ Let $E$ be an injective right $R$-module and $f:E\rightarrow S$ with $S$ is a simple right $R$-module. If $f=0$, there is nothing to prove. We may assume that $f$ is a nonzero homomorphism, and so $f$ is an epimorphism. Since $R$ is right Hereditary, $S$ is injective, and so by (2), $S$ is projective. Hence, the natural epimorphism $\pi:R\rightarrow S$ splits, i.e. there exists a homomorphism $\eta:S\rightarrow R$ such that $\pi\eta=1_{S}$. Then, $\pi\eta f=f$, and so $E$ is max-projective.

$(4)\Rightarrow (5)$ Let $E$ be a singular injective right $R$-module. Assume to the contrary that $E$ has a maximal submodule $K$ such that $E/K\cong R/I$ for some maximal right ideal $I$ of $R$. So, there is a nonzero homomorphism $f:E\rightarrow R/I$, and by (4), there exists a nonzero homomorphism $g:E\rightarrow R$ such that $\pi g=f$, where $\pi:R\rightarrow R/I$ is the canonical epimorphism. Since $E$ is singular, $\Image(g)$ is singular. Moreover, $\Image(g)\subseteq R$, and so $\Image(g)$ is nonsingular. This implies that $g(E)=0$, yielding a contradiction.

$(5)\Rightarrow (6)$ Let $E$ be an injective right $R$-module. Since $R$ is a right nonsingular ring, $Z(E)$ is a closed submodule of $E$, and so $E=Z(E)\oplus F$ for some submodule $F$ of $E$. Then, by $(5)$, $\Rad(Z(E))=Z(E)$.

$(5)\Rightarrow (3)$ Let $E$ be a singular injective right $R$-module. This implies, by our hypothesis, that $\Rad(E)=E$. Let $f:E\rightarrow R/I$ be homomorphism for some right ideal $I$ of $R$. Since $\Rad(E)=E$ and $\Rad(R/I)\neq R/I$, $f:E\rightarrow R/I$ is not an epimorphism. By the right Hereditary assumption, $f(E)$ is injective, and so $f(E)$ is a direct summand of $R/I$. But since $f(E)\subseteq \Rad(R/I)$, we must have $f(E)\ll R/I$. This means, $f(E)=0$, whence $\Hom(E,R/I)=0$ for each right ideal $I$ of $R$. Therefore, $E$ is $R$-projective.
\end{proof}

\begin{proposition} \label{localalmostQF}
Let $R$ be a local right max-$QF$ ring. Then $R$ is either right self-injective or right small.
\end{proposition}
\begin{proof}
Let $J$ be the unique maximal right ideal of $R$ and $E$ be the injective hull of the ring $R$. Assume first that $R$ is not a small ring i.e. $\Rad(E) \neq E$. Then $E$ has a maximal submodule $K$ such that $\frac{E}{K}$ is isomorphic to $\frac{R}{J}$ and denote this isomorphism by $f$. Consider the composition $f\pi$ where $\pi:E \rightarrow \frac{E}{K}$ is the canonical projection. Since $R$ is right max-$QF$, there is a nonzero homomorphism $g:E \rightarrow R$ such that
\[
\xymatrix{
  & E \ar@{.>}[dl]_{g} \ar[d]^{f\pi} &\\
R \ar[r]^h & \frac{R}{J} \ar[r] & 0
}
\] commutes.
Furthermore, $h$ is a small epimorphism and $f \pi$ is an epimorphism, which means $g:E \rightarrow R$ is also an epimorphism and splits. Thus, $E \cong R \oplus T$ for some $T$. Hence, $R$ is a right self injective ring.
\end{proof}

\begin{corollary}
Let $R$ be a commutative semiperfect ring. If $R$ is max-$QF$, then $R=S\times T$ where $S$ is self-injective and $T$ is small.
\end{corollary}
\begin{proof}
Let $R$ be a commutative semiperfect ring, then by \cite[Theorem 23.11]{Lam:AFirstCourseInNoncommutativeRings}, $R=R_{1}\times...\times R_{n}$, where $R_{i}$ is a local ring $(1\leq i\leq n )$. Hence, by Lemma \ref{product} and Proposition \ref{localalmostQF}, $R$ can be written as a direct product of local max-$QF$ rings and every local max-$QF$ ring either self-injective or small.
\end{proof}

\begin{corollary}
Let $R$ be a right Noetherian local ring. Then the following are equivalent.
\begin{enumerate}
\item $R$ is right almost-$QF$.
\item $R$ is right max-$QF$.
\item $R$ is $QF$ or right small.
\end{enumerate}
\end{corollary}
\begin{proof}
$(1)\Rightarrow (2)$ Clear. $(3)\Rightarrow (1)$ Follows from Proposition \ref{smallalmostQF}.

$(2)\Rightarrow (3)$ Clear by Proposition \ref{localalmostQF}, since right Noetherian right self-injective rings are $QF$.
\end{proof}

We do not know whether every right chain ring is almost-$QF$. But the following result will imply that each right chain ring with $P(R)=0$ is right almost $QF$.

\begin{proposition} \label{P(R)=0}
Let $R$ be a right chain ring and $J=J(R)$. Then $P(R)=\bigcap_{n\geq 1} J^{n}$.
\end{proposition}
\begin{proof}
Assume first that $J^{n}=0$ for some $n\in \mathbb{Z}^{+}$. Then $\bigcap_{n\geq 1} J^{n}=0$, and so, by \cite[Proposition 5.3(b)]{module-theory}, $R$ is a right Noetherian ring with $P(R)=0$. On the other hand if we suppose that $J^{n}\neq 0$ for all $n\in \mathbb{Z}^{+}$, then, by \cite[Proposition 5.2(d)]{module-theory}, $A=\bigcap_{n\geq 1} J^{n}$ is a completely prime ideal. Let us now look at the case $A\neq AJ$. Then $\frac{A}{AJ}$ simple right $R$-module and $AJ\ll A$. Let $a\in A \setminus AJ$. If we have $A=aR+AJ$, then $A=aR$, whence either $A=J(R)$ or $A=0$, by \cite[Proposition 5.2(f)]{module-theory}. If $A=\bigcap_{n\geq 1} J^{n}=0$, then $R$ is a right Noetherian ring with $P(R)=\bigcap_{n\geq 1} J^{n}=0$. Otherwise, if $A=J(R)=\bigcap_{n\geq 1} J^{n}$, then $J=J^{2}$, but since $A\neq AJ$, this is not the case. If we look at the case $A=AJ$, then $A\subseteq P(R)$. Since $P(R)=P^{2}(R)$, $P(R)$ is a completely prime ideal of $R$, and so, by \cite[Lemma 5.1]{module-theory}, $P(R)\subseteq A$. Hence, $P(R)=A=\bigcap_{n\geq 1} J^{n}$.
\end{proof}

\begin{corollary} Let $R$ be a right chain ring. Then $R/P(R)$ is a right almost-$QF$ ring.
\end{corollary}
\begin{proof}
Since $P(R)$ is an ideal of $R$, and every factor ring of a right chain ring is a right chain ring, without loss of generality we may assume that $P(R)=0$. Then by Proposition \ref{P(R)=0} and \cite[Proposition 5.3]{module-theory}, $R$ is a right Noetherian ring. We have two cases for $J=J(R)$: if $J$ is nilpotent, then $R$ is Artinian. This implies that $R$ is right self-injective by \cite[Lemma 5.4]{module-theory} which then yields, $R$ is $QF$. So now assume that $J$ is not nilpotent. Then $R$ is a domain by \cite[Proposition 5.2(d)]{module-theory}, whence $R$ is right small. So, $R$ is right almost-$QF$ by Proposition \ref{smallalmostQF}. Thus in any case $R$ is right almost-$QF$.
\end{proof}

We shall characterize commutative Noetherian max-$QF$ rings.

\begin{proposition} (See \cite{matlisinj}) \label{matlis}
Let $R$ be a commutative Noetherian ring, $P$ be a prime ideal of $R$, $E=E(R/P)$, and $A_{i}=\{x\in E : P^{i}x=0\}$. Then:
\begin{enumerate}
\item $A_{i}$ is a submodule of $E$, $A_{i}\subseteq A_{i+1}$, and $E=\bigcup A_{i}$.
\item If $P$ is a maximal ideal of $R$, then $A_{i}\subseteq E(R/P)$ is a finitely generated $R$-module for every integer $i$.
\item $E(R/P)$ is Artinian.
\end{enumerate}
\end{proposition}
\begin{lemma} \label{E(S)projective}
Let $R$ be a commutative Noetherian ring, and let $E=E(R/Q)$ for a maximal ideal $Q$ of $R$. The following are equivalent.
\begin{enumerate}
\item $E$ is $R$-projective.
\item $E$ is max-projective.
\item $\Rad(E)=E$ or $E$ is projective, local and isomorphic to an ideal of $R$.
\end{enumerate}
\end{lemma}
\begin{proof} $(1) \Rightarrow (2)$ is clear.

$(2) \Rightarrow (3)$ Assume that $\Rad(E)\neq E$. Since $R$ is commutative, $\Rad(E)=\bigcap_{i\in \wedge} IE$, where $\wedge$ is the set of all maximal ideals of $R$, \cite[Exercises 15.(5)]{Anderson-Fuller:RingsandCategoriesofModules}. Now we will see that $IE=E$ for any maximal ideal $I$ distinct from $Q$. Let $I$ be a maximal ideal distinct from $Q$. The fact $I+Q=R$ implies $I+Q^{n}=R$ for any $n\in \mathbb{N}$. Let $x\in E$. Then $Q^{n}x=0$ for some $n\in \mathbb{N}$, by Proposition \ref{matlis}. We have  $1=y+z$, where $y\in I$, $z\in Q^{n}$, and then $x=yx\in IE$. Hence, $\Rad(E)=\bigcap_{i\in \wedge} IE=QE\neq E$. Since $R$ is commutative and $(E/QE)Q=0$, $E/QE$ is a semisimple $R/Q$-module, and so $E/QE$ semisimple as an $R$-module. Then $E/QE$ is finitely generated by Artinianity of $E$, and hence $QE+K=E$ for some finitely generated submodule $K$ of $E$. Since $K$ is finitely generated, $K$ is a submodule of $A_{n}$ for some $n$, by Proposition~\ref{matlis}. Thus $Q^{n}K=0$. Since $QE+K=E$ ,  $Q^{n+1}E=Q^{n}E$, implying $Q^{n}E\subseteq P(E)$. On the other hand, $Q^{2}E+QK=QE$, and so $Q^{2}E+K=E$. Continuing in this manner $Q^{n}E+K=E$, whence $E/Q^{n}E$ is finitely generated. Since $R$ is Noetherian, $P(E/Q^{n}E)=0$ and so $P(E)=Q^{n}E$. Since $E/P(E)$ is finitely generated, $E/P(E)$ has finite composition length by Proposition \ref{matlis}(3). By max-projectivity of $E$ and Lemma \ref{factormaxprojective}, $E/P(E)$ is max-projective. Thus $E/P(E)$ is projective by Corollary \ref{cor:finitelengthmaxproj.isproj.}. Then, $E=P(E)\oplus L$ for some projective submodule $L$ of $E$. Since $E$ is indecomposable and $P(E)\neq E$, $E=L$. Therefore $E$ is projective. Furthermore, since $E$ is indecomposable,  the endomorphism ring of $E$ is local by \cite[Lemma 2.25]{module-theory}. By \cite[Theorem 4.2]{ware}, $E$ is a local module, so it is cyclic and $R\cong E \oplus I$ for some ideal $I$ of $R$. Hence $E$ is isomorphic to an ideal of $R$. This proves $(3)$.

$(3) \Rightarrow (1)$ is obvious.
\end{proof}

\begin{lemma} (See \cite[9.7]{kasch})  \label{rad} Suppose $R$ commutative Noetherian or semilocal right Noetherian ring and $\{M_{i}\}_{i\in I}$ be a class of right $R$-modules. Then $\Rad(\prod_{i\in I}M_{i})=\prod_{i\in I}\Rad(M_{i})$.
\end{lemma}

\begin{lemma}
Let $R$ be a commutative Noetherian ring. Then the following are equivalent.
\begin{enumerate} \label{smallring}
\item $R$ is a small ring, i.e., $R\ll E(R)$.
\item $\Rad(E(S))=E(S)$ for each simple $R$-module $S$.
\end{enumerate}
\end{lemma}
\begin{proof}
$(1)\Rightarrow(2)$: Clear by Lemma \ref{small}.

$(2)\Rightarrow(1)$: Let $\Delta$ be a complete set of representatives of simple $R$-modules. Then  $C=\oplus_{S \in \Delta} E(S)$ is an injective cogenerator. Then, for some index set $I$,  the injective hull $E(R)$ of $R$ is a direct summand of $C^I$. By Lemma \ref{rad}, $\Rad(C^I)=C^I$. Since $E(R)$ is a direct summand of $C^I$, we have $\Rad(E(R))=E(R)$. Thus  $R$ is a small ring by Lemma \ref{small}.
\end{proof}
\begin{theorem}
Let $R$ be a commutative Noetherian ring. Then the following are equivalent.
\begin{enumerate}
\item $R$ is almost-$QF$.
\item $R$ is max-$QF$.
\item $R=A \times B$, where $A$ is $QF$ and $B$ is small.
\end{enumerate}
\end{theorem}
\begin{proof}
$(1)\Rightarrow (2)$ is clear.

$(2) \Rightarrow (3)$ First suppose that $\Rad(E(S))=E(S)$ for all simple $R$-module $S$. Then $R$ is a small ring by Lemma \ref{smallring}.
On the other hand, if $\Rad(E(S))\neq E(S)$ for some simple $R$-module $S$, then $E(S)$ is isomorphic to a direct summand of $R$ by Lemma \ref{E(S)projective}. Let $X$ be sum of  minimal ideals $U$ of $R$ with $\Rad(E(U))\neq E(U)$. Then $E(U)$ is isomorphic to an ideal of $R$. Thus without loss of generality we can assume that $E(U)$ is an ideal of $R$. Since $R$ is Noetherian, $X$ is finitely generated, and so $A=E(X)=E(U_1)\oplus \cdots E(U_n)$ where each  $E(U_i)$ is an ideal of $R$.   Thus $R=A\oplus B$ for some ideal $B$ of $R$. Now $A$ is injective and Noetherian, so $A$ is a $QF$ ring. On the other hand, let $V$ be a simple $B$-module, then $V$ is a simple $R$-module. Let $E(V)$ be the injective hull of $V$.  As $V$ is a $B$-module, $VA=0$. If $\Rad(E(V))\neq E(V)$, then this would imply $V\subseteq A$, by the same arguments above. Thus  $\Rad(E(V))=E(V)$, and so $B$ is a small ring by Lemma \ref{smallring}.

$(3)\Rightarrow (1)$ Clear, by Proposition \ref{smallalmostQF} and Lemma \ref{product}.
\end{proof}

\begin{proposition}
Let $R$ be a semiperfect ring. Then the following are equivalent.
\begin{enumerate}
\item $R$ is right  almost-$QF$ and direct sum of small right $R$-modules is small.
\item $R$ is right max-$QF$ and direct sum of small right $R$-modules is small.
\item $R$ is right almost-$QF$ and $\Rad(Q)\ll Q$ for each injective right $R$-module $Q$.
\item $R$ is right max-$QF$ and $\Rad(Q)\ll Q$ for each injective right $R$-module $Q$.
\item $R$ is $QF$.
\end{enumerate}
\end{proposition}
\begin{proof}
$(1)\Rightarrow (2)$ and $(3)\Rightarrow (4)$ Clear. $(2)\Rightarrow(3)$ and $(3)\Rightarrow(1)$ By \cite[Lemma 9]{MR653852}.

$(4)\Rightarrow (5)$ Let $M$ be an injective right $R$-module. Since $M$ is max-projective with $\Rad(M)\ll M$, by Proposition \ref{maxprojectiveradsmall}$(1)$, $M$ is projective. Hence $R$ is $QF$.

$(5)\Rightarrow (3)$ Let $M$ be an injective right $R$-module. By the hypothesis, $M$ is projective. Since $R$ is right Artinian, every right $R$-module has a small radical, whence $\Rad(M)\ll M$.
\end{proof}

In \cite{coneat}, a submodule $N$ of a right $R$-module $M$ is called \emph{coneat} in $M$ if $\Hom(M,S)\rightarrow \Hom(N,S)$ is epic for every simple right $R$-module $S$. In \cite{spure}, $N$ is called \emph{s-pure} in $M$ if $N\otimes S\rightarrow M\otimes S$ is monic for every simple left $R$-module $S$. M is \emph{absolutely coneat} (resp., \emph{absolutely s-pure}) if $M$ is coneat (resp., s-pure) in every extension of it. If $R$ is commutative, then s-pure short exact sequences coincide with coneat short exact sequences, \cite[Proposition 3.1]{FuchsNeat}.
\begin{proposition}
Consider the following conditions for a ring R:
\begin{enumerate}
\item $R$ is right max-$QF$.
\item Every absolutely coneat right $R$-module is max-projective.
\item Every absolutely s-pure right $R$-module is max-projective.
\item Every absolutely pure right $R$-module is max-projective.
\end{enumerate}
Then $(3)\Rightarrow (4)\Rightarrow (1)\Rightarrow (2)$. Also, if $R$ is a commutative ring, then $(2)\Rightarrow (3)$.
\end{proposition}
\begin{proof}
$(3)\Rightarrow (4)\Rightarrow (1)$ Clear.

$(1)\Rightarrow (2)$ Let $M$ be an absolutely coneat right $R$-module. Consider the following diagram:
$$\xymatrix{0 \ar[r]  &M \ar[d]^{f} \ar[r]^{i}&E(M)\\
R\ar[r]^{\pi} & S \ar[r] & 0
} $$ where $S$ is a simple right $R$-module, $i:M\rightarrow E(M)$ is the inclusion map and $\pi:R\rightarrow S$ is the canonical quotient map. Since $M$ coneat in $E(M)$, there is a homomorphism $g:E(M)\rightarrow S$ such that $gi=f$. Also, by $(1)$, there exists a homomorphism $h:E(M)\rightarrow R$ such that $\pi h=g$. Hence, $(\pi h)i=gi=f$.

$(2)\Rightarrow (3)$ Let $M$ be an absolutely s-pure right $R$-module. Then $M$ is s-pure in $E(M)$. Since $R$ is commutative, $M$ is coneat in $E(M)$. Hence, $M$ is max-projective by $(2)$.
\end{proof}

In \cite[Lemma 1.16]{Nicholson}, it was shown that for a projective module $M$, if $M = P + K$, where $P$ is a summand of $M$ and $K \subseteq M$, then there exists a submodule $Q \subseteq K$ with $M = P \oplus Q$. By using the same method in the proof of \cite[Theorem 2.8]{almostperfect}, one can prove the following result.
\begin{proposition}
A ring $R$ is right almost-$QF$ if and only if for every injective right $R$-module $E$, if $E=P+L$, where $P$ is a finitely generated projective summand of $E$ and $L\subseteq E$, then $E=P\oplus K$ for some $K\subseteq L$.
\end{proposition}

Let R be a ring and $\Omega$ be a class of $R$-modules which is closed under isomorphic copies. Following Enochs, a homomorphism $\varphi:G\rightarrow M$ with
$G\in \Omega$ is called an $\Omega$-precover of the $R$-module $M$ if for each homomorphism $\psi:H\rightarrow M$ with $H\in \Omega$, there exists $\lambda:H\rightarrow G$ such that $\varphi\lambda=\psi$.
\begin{lemma}
Let $R$ be a right self-injective ring. Then the following are equivalent.
\begin{enumerate}
\item $R$ is right almost-$QF$.
\item Every finitely generated right $R$-module has an injective precover which is $R$-projective.
\item Every cyclic right $R$-module has an injective precover which is $R$-projective.
\end{enumerate}
\end{lemma}
\begin{proof}
$(1)\Rightarrow (2)$ Let $M$ be a finitely generated right $R$-module and $g:R^{n}\rightarrow M$ be an epimorphism. For any homomorphism $f:E\rightarrow M$ with $E$ is injective, there exists $h:E\rightarrow R^{n}$ such that $gh=f$. Since $R^{n}$ is injective, $g$ is an injective precover of $M$.

$(2)\Rightarrow (3)$ Clear.

$(3)\Rightarrow (1)$ Let $E$ be an injective right $R$-module and $I$ be a right ideal of $R$. Suppose that $f:E\rightarrow R/I$ is a homomorphism, $\pi:R\rightarrow R/I$ is the natural epimorphism and $g:G\rightarrow R/I$ be an injective cover of $R/I$. So, there is $h:E\rightarrow G$ such that $gh=f$. By hypothesis, $G$ is $R$-projective and hence there is $k:G\rightarrow R$ such that $\pi k=g$. Let $\overline{f}=kh$. So $\pi\overline{f}=\pi kh=gh=f$. Therefore, E is $R$-projective, and so $R$ is right almost-$QF$.
\end{proof}
In \cite{copureinjective}, a module $M$ is said to be \emph{copure-injective} if $\Ext^{1}_{R}(E,M)=0$ for any injective module $E$. Now we give the characterization of almost-$QF$ rings in terms of copure-injective modules.

\begin{proposition}
Let $R$ be a ring. Then the followings are equivalent.
\begin{enumerate}
\item $R$ is right almost-$QF$ and $R_{R}$ is copure-injective.
\item Every right ideal of $R$ is copure-injective.
\item Every submodule of a finitely generated projective right $R$-module is copure injective.
\end{enumerate}
\end{proposition}
\begin{proof}
$(1)\Rightarrow (2)$ Let $E$ be an injective right $R$-module and $I$ be a right ideal of $R$. By applying $\Hom(E,-)$ to the short exact sequence $0\rightarrow I\rightarrow R\rightarrow R/I\rightarrow 0$, we obtain the following exact sequence: $0\rightarrow \Hom(E,I)\rightarrow \Hom(E,R)\rightarrow \Hom(E,R/I)\rightarrow \Ext^{1}_{R}(E,I)\rightarrow \Ext^{1}_{R}(E,R)\rightarrow ...$. Since $R_{R}$ is copure-injective, $\Ext^{1}_{R}(E,R)=0$. Then the map $\Hom(E,R)\rightarrow \Hom(E,R/I)$ is onto since $E$ is $R$-projective. Hence, $\Ext^{1}_{R}(E,I)=0$ for any injective $R$-module $E$.

$(2)\Rightarrow (3)$ Suppose that every right ideal of $R$ is copure-injective. First, by induction, we show that every submodule of $R^{n}$ is copure-injective. The case $n=1$ follows by the hypothesis. Now suppose that $n> 1$ and every submodule of $R^{n-1}$ is copure-injective. Let $N$ be a submodule of $R^{n}$, and consider the exact sequence $0\rightarrow N\cap R^{n-1}\rightarrow N\rightarrow N/(N\cap R^{n-1})\rightarrow 0$. By induction hypothesis, $N\cap R^{n-1}$ is copure-injective, and $N/(N\cap R^{n-1})\cong (N + R^{n-1})/R^{n-1}\subseteq R^{n}/R^{n-1}\cong R$ is also copure-injective. Therefore, for any injective right $R$-module $E$, consider the exact sequence $ \Ext^{1}_{R}(E, N\cap R^{n-1})\rightarrow \Ext^{1}_{R}(E,N)\rightarrow \Ext^{1}_{R}(E,N/(N\cap R^{n-1}))$. Since $\Ext^{1}_{R}(E, N\cap R^{n-1})=\Ext^{1}_{R}(E,N/(N\cap R^{n-1}))=0$, we have $\Ext^{1}_{R}(E,N)=0$. Therefore, $N$ is copure-injective. Now if $M$ is a submodule of a finitely generated projective right $R$-module $P$, then there is $n\geq 1$ such that $M\subseteq P\subseteq R^{n}$. By the above observation, $M$ is also copure-injective.
$(3)\Rightarrow (2)$ is clear. $(2)\Rightarrow (1)$ by Proposition \ref{p-testing}.
\end{proof}
\begin{proposition}
Let $R$ be a ring. Then the following are equivalent.
\begin{enumerate}
\item $R$ is semisimple.
\item $R$ is right almost-$QF$ right V-ring.
\item $R$ is right almost-$QF$ and every submodule of an $R$-projective right module is $R$-projective.
\item $R$ is right self-injective and every submodule of an $R$-projective right module is $R$-projective.
\end{enumerate}
\end{proposition}
\begin{proof}
$(1)\Rightarrow (2)$, $(1)\Rightarrow (3)$ and $(1)\Rightarrow (4)$ are clear.

$(2)\Rightarrow (1)$ Since $R$ is a right $V$-ring, every simple right $R$-module is injective. By the hypothesis, every simple right $R$-module is $R$-projective, whence projective.

$(4)\Rightarrow (1)$ Let $M$ be a cyclic right $R$-module and $I$ a right ideal of $R$. Consider the following diagram:
 $$\xymatrix{0 \ar[r]  &I \ar[d]^{f} \ar[r]^{i}&R\\
R\ar[r]^{\pi} & M \ar[r] & 0
} $$ where $i:I\rightarrow R$ is the inclusion map and $\pi:R\rightarrow M$ is the canonical quotient map. Since $I$ is $R$-projective there exists $h:I\rightarrow R$ such that $\pi h=f$. By the injectivity of $R$, there exists $\lambda:R\rightarrow R$ such that $\lambda i=h$. Then $(\pi\lambda)i=\pi h=f$, and $\pi\lambda:R\rightarrow M$ is the required map.

$(3)\Rightarrow (1)$ Since every simple right $R$-module can be embedded in an injective $R$-module, every simple right $R$-module is $R$-projective, and so every simple right $R$-module is projective. Hence, $R$ is semisimple.
\end{proof}

\end{document}